\documentclass[12pt]{article}
\usepackage{amssymb,amsmath,graphics}
\usepackage[pdftex]{graphicx,color}
\usepackage{longtable}
\usepackage{cite}
\usepackage{amsfonts}
\usepackage{amsthm}
\usepackage{times}
\usepackage{epsfig}
 
\title{The Enumerative Geometry of Hyperplane Arrangements}

\author{
Thomas Paul, Will Traves and Max Wakefield \thanks{Department of Mathematics, U.S. Naval Academy,
traves@usna.edu, wakefiel@usna.edu}}

\date{}

\theoremstyle{plain}
\newtheorem{theorem}{Theorem}

\newtheorem{lemma}[theorem]{Lemma}

\theoremstyle{definition}
\newtheorem{definition}[theorem]{Definition}
\newtheorem{example}[theorem]{Example}
\newtheorem{remark}[theorem]{Remark}

\newtheorem{question}[theorem]{Question}

\newcommand{\A}{\mathcal{A}}
\newcommand{\M}{\mathcal{M}}
\newcommand{\N}{N}
\newcommand{\NN}{\mathbb{N}}
\renewcommand{\L}{\mathcal{L}}
\newcommand{\MLA}{\M_{\A}}
\newcommand{\C}{\mathbb{C}}
\newcommand{\Z}{\mathbb{Z}}

\newcommand{\G}{\mathbb{G}}

\renewcommand{\P}{\ensuremath \mathbb{P}}
\renewcommand{\S}{\ensuremath \mathcal{S}}

\newcommand{\codim}{\text{codim}}

\begin{document}
\maketitle

\begin{abstract}

We study enumerative questions on the moduli space $\M(L)$ of hyperplane arrangements with a given intersection lattice $L$. Mn\"ev's universality theorem suggests that these moduli spaces can be arbitrarily complicated; indeed it is even difficult to compute the dimension $D =\dim \M(L)$. Embedding $\M(L)$ in a product of projective spaces, we study the degree $N=\deg \M(L)$, which can be interpreted as the number of arrangements in $\M(L)$ that pass through $D$ points in general position. For generic arrangements $N$ can be computed combinatorially and this number also appears in the study of the Chow variety of zero dimensional cycles. We compute $D$ and $N$ using Schubert calculus in the case where $L$ is the intersection lattice of the arrangement obtained by taking multiple cones over a generic arrangement. We also calculate the characteristic numbers for families of generic arrangements in $\P^2$ with 3 and 4 lines.

\end{abstract}

\section{Introduction}

Enumerative geometry has a monumental history and continues to be an inspiration for many different fields of research (for example see Katz \cite{Katz} and Kontsevich and Manin \cite{KM-94}). Solutions to enumerative problems have given deep insight into the geometric nature of various algebraic varieties and important spaces.

Realization spaces of matroids (or moduli spaces of hyperplane arrangements) give a beautiful connection between combinatorics and algebraic geometry. In particular, Mn\"{e}v's universality theorem presents matroids as a kind of dictionary for quasi-projective varieties (see Mn\"{e}v \cite{Mnev} or Vakil \cite{VakilMurphy}). These realization spaces have been studied from many different viewpoints. Kapranov \cite{Kapranov-93} showed the Hilbert compactification of the moduli space of generic arrangements could be viewed as a Chow quotient. Hacking, Keel, and Tevelev \cite{HKT-06} applied the relative minimal model program to the moduli space of arrangements. Terao \cite{Terao-02} studied the closure of this moduli space in a product of projective spaces and an associated logarithmic Gauss-Manin connection. Speyer \cite{Speyer-09} proved that the $K$-theory class (actually a push forward-pullback) of the inclusion of the moduli space into the appropriate Grassmannian is actually the 2-variable Tutte polynomial of the associated matroid. Then Fink and Speyer \cite{Fink-Speyer-12} generalized this result to non-realizable matroids. In this note we study some of the geometry of this moduli space and find another use of the Tutte polynomial.

One of our motivating problems is Terao's conjecture which concerns the subset of the realization consisting of free arrangements (see Orlik and Terao \cite{OT} for a general reference on hyperplane arrangements). Yuzvinsky \cite{Yuz-93} showed that this subset was Zariski-open. It is not known if this subset is also closed. Our focus here is on computing the dimension and essentially the degree of this realization space by using both combinatorial and geometric methods.

To each hyperplane arrangement $\A = \{H_1, \ldots, H_k\}$ in $\P^ n$
we associate an {\em intersection lattice} $\L(\A)$, the poset whose
elements are intersections $H_{i_1} \cap \cdots \cap H_{i_t}$ ordered
by reverse inclusion, $B \leq C \iff C \subseteq B$. Two arrangements
$\A$ and $\A'$ are {\em combinatorially equivalent} if their
intersection lattices are isomorphic, that is, if there is a bijection
$\phi: \A \rightarrow \A'$ that preserves the lattice order, $B \leq C
\iff \phi(B) \leq \phi(C)$. Let $\M_\A$ be the set of arrangements
that are combinatorially equivalent to the arrangement
$\A$. Identifying each arrangement $\{H_1, \ldots, H_k\}$ with its orbit
$\{(H_{\sigma(1)},\ldots, H_{\sigma(k)}):\; \sigma \in \S_k\}$ under the permutation group
$\S_k$, we obtain an embedding $\M_\A \hookrightarrow
(\P^{n*})^k/\S_k$ and we give $\M_\A$ the induced topology coming from
the Zariski topology on the quotient space. It is clear that $\M_\A$ depends only on the intersection lattice of $\A$. If $D = \dim
\M_\A$ is the dimension of $\M_\A$, then let $N_\A$ be the number of
arrangements in $\M_A$ passing through $D$ points in general
position in $\P^n$.

\begin{question} \label{mainq} Given $\A$ compute $D$ and $\N_\A$. Ideally these
  answers should be given in terms of the combinatorics of the lattice
  $\L(\A)$.
\end{question}

In the next section we show that when $\A$ is a generic arrangement of
$k$ hyperplanes in $\P^n$, then $\dim \M_\A = kn$ and we compute
$\N_\A$. The characteristic number $N_\A(p,\ell)$ measures the
number of arrangements combinatorially equivalent to $\A$ that pass
through $p$ points and are tangent to $\ell$ lines in general position
(with $p + \ell = \dim \M_\A$). We use intersection theory on the correspondence
between hyperplane arrangements and their dual arrangements to compute
the characteristic numbers for generic arrangements of three or four
lines in $\P^2$ (a good reference for the intersection theory that we
use is Eisenbud and Harris \cite{EH} or Fulton \cite{Fulton}). This seems to be the first computation
of these characteristic numbers for line arrangements though
characteristic numbers were computed for smooth curves of degrees 3
and 4 by Zeuthen \cite{Zeuthen} in the 19$^\text{th}$ century. As reported in Kleiman \cite{Kleiman15}, the
19$^\text{th}$ century methods lacked adequate foundations prompting
Hilbert to ask for a rigorous computation of these characteristic
numbers. Aluffi \cite{Aluffi} and Kleiman and Speiser
\cite{KSp} verified Zeuthen's degree 3 predictions using intersection
theory. Vakil \cite{Vakil} used intersection theory on the moduli
space of stable maps to verify the degree 4 predictions. The
importance of the characteristic numbers is suggested by a theorem
originally due to Zeuthen \cite{Zeuthenbook} (also see Fulton
\cite[section 10.4]{Fulton}) that shows that the characteristic
numbers for a family of curves determine the number of such curves
tangent to smooth curves of {\em arbitrary} degrees. We close Section
\ref{section:generic} by interpreting this theorem for line arrangements. It would be interesting to use intersection theory on the moduli space of stable maps to recover our results.

In Section \ref{section:dconed} we consider cones over generic arrangements of
hyperplanes.

\begin{definition} \label{dconed} A $d$-coned generic arrangement $\A$
  in $\P^n$ is an arrangement of $k > n$ hyperplanes obtained from a
  generic hyperplane arrangement $\mathcal{B}$ in a linear subspace $H
\cong \P^{n-(d+1)}$ by taking the cone with a $d$-dimensional linear
  space (equivalently, with $d+1$ general points). That is, there
  exists a linear space $\Omega$ of dimension $d$, disjoint from $H$
  so that each hyperplane in $\A$ is the linear span of both a
  hyperplane in $\mathcal{B}$ and $\Omega$. \end{definition}

In Section \ref{section:dconed} we answer the enumerative problems from Question
\ref{mainq} for $d$-coned generic arrangements. Here the methods of
Schubert Calculus come into play and the Catalan numbers make a cameo
appearance.

Our approach to computing $\N_A$ is to first compute the number of
{\em labeled} arrangements with intersection lattice isomorphic to $\L(\A)$
that pass through $D$ points in general position in $\P^n$. Dividing by
the number of ways to label the hyperplanes in $\A$ gives $\N_A$. This
allows us to work in a product of polynomial rings rather than its
quotient. A recent paper by Feh{\'e}r, N{\'e}methi, and Rim{\'a}nyi
\cite{Rimanyi} also studies enumerative problems involving hyperplane
arrangements; they embrace quotient varieties and work with equivariant
cohomology.

Though we were not able to answer Question \ref{mainq} in terms of the
combinatorics of $\L(\A)$, we remain optimistic about this possibility
for special families of arrangements despite several warning signs that the question may be very
difficult in general. Mn\"{e}v's Universality Theorem says that each variety appears as the
closure of $\M_\A$ for some hyperplane arrangement $\A$, so the
geometry of $\M_\A$ can be arbitrarily complicated. Another warning
sign appears if we take a naive approach to the dimension problem in
$\P^2$. For each line arrangement $\A$ with $\ell$ labeled lines
$L_1, \ldots, L_\ell$ and $k$ labeled points $p_1, \ldots, p_k$ of
intersection among these lines, form the parameter space
$$ P_\A = \{ (H_1, \ldots, H_\ell, P_1, \ldots, P_k) \in (\P^{2*})^\ell \times
(\P^2)^k \; : \; P_i \in H_j \text{ if $p_i \in L_j$ }\}. $$ Note that
$\dim P_{\A} = \dim \M_\A$. If the conditions $P_i \in H_j$ are
algebraically independent, then this dimension is $2(k+\ell) -
\sum_{i=1}^k |\{ L_j : p_i \in L_j\}|$ (sometimes this is called the
virtual dimension of $\M_\A$). Such a formula holds for
generic line arrangements and pencils as well as for many other
arrangements; however, it fails for the Pappus arrangement pictured in
Figure \ref{figpappus}, since one
of the conditions $P_i \in H_j$ is implied by the others. Indeed, the
dimension of $\M_\A$ depends on the syzygies among these incidence
conditions and so any potential algorithms need to be sensitive enough
to recognize such syzygies from the combinatorial information in
$\L(A)$, a task that appears to be quite difficult. This example (more
precisely, its projective dual) was studied in detail by Feh{\'e}r,
N{\'e}methi, and Rim{\'a}nyi \cite{Rimanyi} and by Ren, Richter-Gebert and Sturmfels \cite{Ren}.

\begin{figure}[h!t]
\begin{center}
\scalebox{0.5}{\includegraphics{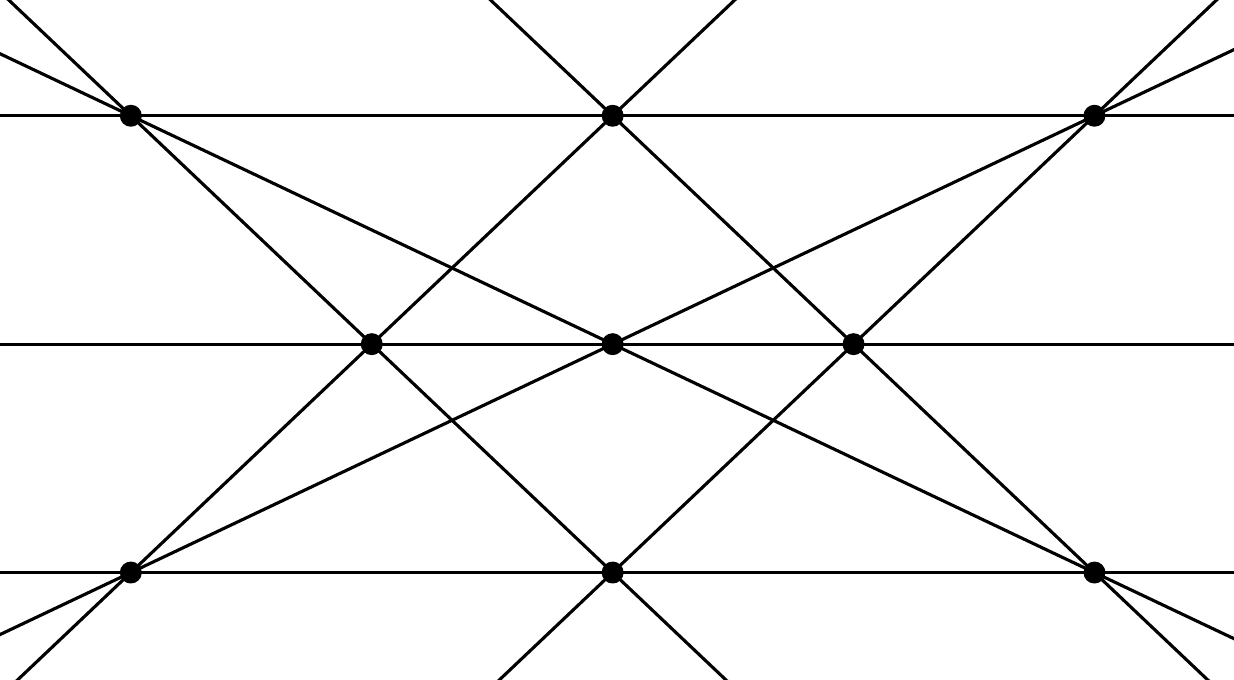}}
\end{center}
\caption{The Pappus arrangement consists of nine lines. The 27
  point-line incidence conditions are dependent: the incidence
  correspondence has codimension 26.}
\label{figpappus}
\end{figure}

\begin{remark} \label{notation} We use the notation $\{(a)^{a'},(b)^{b'}, \ldots,
(k)^{k'}\}$ (sometimes omitting the braces) to refer to the multiset
  consisting of $a'$ copies of $a$, $b'$ copies of $b$, et
  cetera. Similarly, we use $\binom{d}{(a)^{a'},(b)^{b'}, \ldots,
(k)^{k'}}$ to refer to the multinomial that counts the ways of choosing $a'$ groups of
$a$ objects, $b'$ groups of $b$ objects, $\ldots$, and $k'$ groups of
$k$ objects from $d$ labeled objects. That is, $\binom{d}{(a)^{a'},(b)^{b'}, \ldots,
(k)^{k'}} = d!/(a!)^{a'} \cdots (k!)^{k'}$. \end{remark}

\section{Generic Arrangements} \label{section:generic}

An arrangement of $k$ hyperplanes in $\P^n$ is said to be {\em generic} if
$k > n$ and no point is in the intersection of more than $n$ of the
hyperplanes. Carlini discovered the following fact while studying the
Chow variety of zero dimensional degree-$k$ cycles in $\P^n$
\cite[Proposition 3.4]{Carlini}.

\begin{theorem}\label{generics} When $\A$ is a generic arrangement of $k$ hyperplanes
  in $\P^n$ then the dimension $D$ of $\MLA$ is $kn$ and the number of
  arrangements with lattice type isomorphic to $\L(A)$ that pass through
  $D$ points in general position in $\P^n$ is $$ \N_\A = \frac{1}{k!}
\binom{kn}{n} \binom{(k-1)n}{n} \cdots \binom{n}{n} =
\frac{(kn)!}{k!(n!)^k}.$$
\end{theorem}

\begin{proof}
  The moduli space $\MLA$ is a $k!$ to 1 cover of the complement of a
  closed set in $(\P^{n*})^k$ that parameterizes the sets of $k$
  hyperplanes that contain a set of $n$ linearly dependent
  hyperplanes. So $D=\dim \MLA = \dim (\P^{n*})^k = kn$. We count the
  ordered (or labeled) generic arrangements passing through $kn$
  points in general position. Since the points are in general
  position, no more than $n$ can lie on any one hyperplane. So by the
  pigeon-hole principle, each hyperplane contains precisely $n$ of the
  points and each hyperplane is completely determined by these $n$
  points. There are clearly $\prod_{i=0}^{k-1} \binom{kn-in}{n}$
  ways to distribute the points among the labeled hyperplanes, but each
  hyperplane arrangement can be labeled in $k!$ ways so dividing the
  product of binomial coefficients by $k!$ gives $\N_\A$.
\end{proof}

\begin{remark}
We remark that when $n=2$ the formula for $\N_\A$ reduces to
$(2k-1)!!$. Aside from the nice simplicity of this result, this
formulation allows us to interpret the result in terms of the
multivariate Tutte polynomial of the lattice $\L(\A)$. The
multivariate Tutte polynomial of $\L(\A)$ (see Ardila \cite{Ardila} or Sokal
\cite{Sokal}) is defined as
$$ Z_{\L(\A)}(q,x_1,\ldots,x_k) = \sum_{B \in \L(\A)} (q^{-rk(B)})
\prod_{H_j \in B} x_j, $$ where if $B \subset
\A$ then $rk(B) = \codim \cap_{H\in B} H.$ When $\A$ is
a generic arrangement with $k$ hyperplanes, $$Z_{\L(\A)}(1,x_1, \ldots,
x_k) = \sum_{B \subset \A} \prod_{H_j \in B} x_j =
(1+x_1)(1+x_2)\cdots (1+x_k).$$ In
particular, $$ Z_{\L(\A)}(1,0,2,\ldots,2k-2) = (2k-1)!! = \N_\A. $$
\end{remark}

\begin{remark} The intersection ring (or Chow ring) $A(X)$ of a variety $X$ can be defined as the ring of equivalence classes of algebraic subvarieties of $X$ modulo rational equivalence, graded by codimension \cite{Fulton}. Here $[Y_1 \cup Y_2] = [Y_1] + [Y_2]$. Also, $[Y_1 \cap Y_2] = [Y_1]\cdot [Y_2]$ if $Y_1$ and $Y_2$ intersect transversely. The intersection rings that we consider can also be interpreted in terms of cohomology; for more on this point of view, see Katz \cite{Katz}.

When $X\cong \P^{n_1}\times \cdots \times \P^{n_k}$ is a product of projective spaces then $$A(X) \cong \frac{\mathbb{Z}[x_1,\ldots,x_k]}{(x_1^{n_1+1},\ldots,x_k^{n_k+1})},$$ where the $x_i$ are the pullbacks of the classes of hyperplanes on each factor to the product. \end{remark}
 
Next we prove a Lemma that will be used in many of the following arguments.

\begin{lemma}\label{transpts}

Fix an arrangement $\A$ in $\P^n$ with $D=\dim \M_\A$ and $|\A|=k$. The conditions that the arrangement contain $D$ specified points in general position are transverse and the class corresponding to the intersection of these conditions is $$\left(\sum_{i=1}^kx_i\right)^D.$$

\end{lemma}

\begin{proof} The class of the condition that a point lies on the $i$th hyperplane is just $x_i$. Then the class of the condition that a point lies on one of the $k$ hyperplanes is $\sum_{i=1}^kx_i$. Now, the action of $\mathrm{PGL}(n)^k$ is transitive on $\M_\A$. Since the $D$ point conditions are assumed to be generic we can use this $\mathrm{PGL}(n)^k$ action to move one point condition to another. Hence Kleiman's Transversality Theorem \cite{Kleiman} says that these conditions are transverse and the corresponding Chow class is the product of the classes. \end{proof}

Note that Theorem \ref{generics} could also be easily proved by Lemma \ref{transpts}. Since we are not putting any conditions on the generic lines we have nothing but point conditions. Then the degree of the Chow class given in Lemma \ref{generics} counts the number of ordered generic hyperplane arrangements passing through $D$ points in general position. Dividing by $k!$ gives the number of such unordered generic arrangements.

Continuing to focus on the case where $n=2$, we will compute the
characteristic numbers of generic line arrangements with $k=3$ or
$k=4$ lines. We define the characteristic number $N_\A(p,\ell)$ to be the number of arrangements with the same lattice type as $\A$ that
pass through $p$ points and are tangent to $\ell$ lines in general
position (with $p + \ell = \dim \M_\A$). When $\A$ is a generic arrangement of $k$ lines, we simplify the notation to $N_k(p,\ell)$. To get these numbers we compute the degree of a related variety $M$ which is a delicate intersection calculation. We do this in two different ways. First for the case $k=3$ we preform an explicit calculation using the method of undetermined coefficients. Then for the case $k=4$ we argue that the corresponding intersection class is a product using transversality considerations.

\begin{remark} \label{tangencyremark} The only way that an arrangement of 3 lines $\A=\{\ell_1,\ell_2,\ell_3\}$ can be tangent to a given line $L$ is if the cubic defining the arrangement meets the line $L$ at a point of multiplicity strictly greater than 1. If the arrangement is generic this means that an intersection point $\ell_i\cap \ell_j$ of two of the 3 lines must lie on $L$. This motivates the introduction of the points $p_{ij}=\ell_i\cap \ell_j$ in the following lemma: a generic arrangement is tangent to a line $L$ if and only if there exists a point $p_{ij}$ on $L$. \end{remark}

\begin{lemma}\label{3linesM}
The set
$$ M = \left\{(\ell_1,\ell_2,\ell_3,p_{12},p_{13},p_{23}) \in
\left( \P^{2*}\right)^3 \times \left( \P^2 \right)^3\; : \; p_{ij} =
\ell_i \cap \ell_j \text{ for $1 \leq i < j \leq 3$} \right\}$$
is a quasi-projective variety of dimension 6 in $(\P^{2*})^3\times (\P^{2})^3$. The Chow ring is $$A=A((\P^{2*})^3\times (\P^{2})^3)\cong \Z[x_1,x_2,x_3,y_{12},y_{13},y_{23}]/(x_1^3,x_2^3,x_3^3,y_{12}^3,y_{13}^3,y_{23}^3)$$ where the $x_i$ are the classes of the lines and the $y_{ij}$ are the classes of the points. The class of $M$ in $A$ is $$[M]=
\prod_{1 \leq i < j \leq 3} (x_i + y_{ij})(x_j + y_{ij}) .$$

\end{lemma}

\begin{proof} Since $M$ has codimension 6 the class $[M]$ is a degree 6 polynomial in the Chow ring $A$. We write this class as the sum of all degree 6 monomials in $A$: $$[M]=\sum\limits_{a_1+\cdots +a_6=6}c_{(a_i)}x_1^{a_1}x_2^{a_2}x_3^{a_3}y_{12}^{a_4}y_{13}^{a_5}y_{23}^{a_6} .$$ Now we determine all the coefficients. First we show which coefficients are zero. To do this we first partition the set of degree-6 monomials that appear in the expanded product into types. In the definitions that follow the notation $p \in A_i/I$ denotes a degree-$i$ monomial in the variables appearing in the quotient ring $A/I$.

The first type of monomials we define are $$T_1=\left\{ \{ x_1^2y_{12}^2p| p\in A_2/(x_1,y_{12})\} \bigcup \{ x_2^2y_{12}^2p| p\in A_2/(x_2,y_{12})\} \right. $$ $$ \bigcup \{ x_1^2y_{13}^2p| p\in A_2/(x_1,y_{13})\} \bigcup \{ x_3^2y_{13}^2p| p\in A_2/(x_3,y_{13})\} $$ $$\left. \bigcup \{ x_2^2y_{23}^2p| p\in A_2/(x_2,y_{23})\}
\bigcup \{ x_3^2y_{23}^2p| p\in A_2/(x_3,y_{23})\}
\right\} .$$ The second type we define are $$T_2=\left\{ \{x_1^2x_2^2y_{12}p|p\in A_1/(x_1,x_2,y_{12})\} \bigcup \{x_1^2x_3^2y_{13}p|p\in A_1/(x_1,x_3,y_{13})\} \right. $$ $$ \left. \bigcup \{x_2^2x_3^2y_{23}p|p\in A_1/(x_2,x_3,y_{23})\} \right\} .$$ The third type we define are $$T_3=\left\{ \{x_1y_{12}^2y_{13}^2p|p\in A_1/(x_1,y_{12},y_{13})\} \bigcup \{x_2y_{12}^2y_{23}^2p|p\in A_1/(x_2,y_{12},y_{23})\} \right. $$ $$ \left. \bigcup \{x_3y_{23}^2y_{13}^2p|p\in A_1/(x_3,y_{23},y_{13})\} \right\} .$$ The fourth type we define is $$T_4=\{ x_1^2y_{23}^2x_2y_{12},x_1^2y_{23}^2x_3y_{13},x_2^2y_{13}^2x_1y_{12},x_2^2y_{13}^2x_3y_{23},x_3^2y_{12}^2x_1y_{13},x_3^2y_{12}^2x_2y_{23}\} .$$ Note that $|T_1|=54$, $|T_2|=9$, $|T_3|=9$, and $|T_4|=6$.

For any monomial $m\in A_6$ we define the complement of $m$ by the unique monomial $c(m)=m'$ such that $mm'\neq 0 \in A_{12}$. Set $Z_1=T_1\cup T_2\cup T_3\cup T_4$. Notice that $Z_1$ has the property that $c(Z_1)=Z_1$. Hence to show that the monomials in $Z_1$ have the coefficient of zero in $[M]$ we just check that multiplying $[M]$ by each of the monomials in $Z_1$ gives zero. Also, notice that within each type there is a union of smaller sets where each of those sets can be obtained from one another by a permutation of the coordinates. Hence if we can show that the coefficient is zero for one of the subsets then we have completed the task for the entire type. In the following arguments we use duality on $\P^2$ often.

First we argue $[M]m=0$ where $m\in T_1$. Let $m\in T_1$ be in the first subset so $m=x_1^2y_{12}^2m'$. Then in $A$ this monomial $m$ represents fixing a generic line for $\ell_1$ in the first factor of $\P^{2*}$ and a generic point for $p_{12}$. Set $\pi_i : \left( \P^{2*}\right)^3 \times \left( \P^2 \right)^3\to \P^{2*}$ to be the projection to the factor corresponding to the line $\ell_i$ and $\pi_{ij} : \left( \P^{2*}\right)^3 \times \left( \P^2 \right)^3\to \P^{2}$ to be the projection to the factor corresponding to the point $p_{ij}$. With this notion we can say $m=[\pi_{1}^{-1}(\text{generic point}) \cap \pi_{12}^{-1}(\text{generic point})\cap (m' \text{ conditions})]$ where we will not need the conditions from $m'$. Using the Moving Lemma (see \cite[Theorem 5.4]{EH}) $$[M]m=[M\cap \pi_{1}^{-1}(\text{generic point}) \cap \pi_{12}^{-1}(\text{generic point})\cap (m' \text{ conditions})] .$$ In $M$ the point $p_{12}$ must lie on $\ell_1$ and this cannot happen if $\ell_1$ and $p_{12}$ are generic. Hence $[M]m=0.$

The other types have similar arguments. We include them for completeness, but with briefer arguments. For type 2 we look again at the first subset and put $m=x_1^2x_2^2y_{12}m'$. Then $[M]m$ equals $$[M \cap \pi_{1}^{-1}(\text{generic point})\cap \pi_{2}^{-1}(\text{generic point}) \cap \pi_{12}^{-1}(\text{generic line})\cap (m' \text{ conditions})].$$ This means that we have fixed generic lines for lines 1 and 2. Hence the point $p_{12}$ is fixed using $M$, but $p_{12}$ must also lie on a fixed line. We cannot find such a configuration; hence $[M]m=0$.

For type 3 we again look at the first subset $m=x_1y_{12}^2y_{13}^2m'$. Then $[M]m$ equals $$[M \cap \pi_{1}^{-1}(\text{generic line})\cap \pi_{12}^{-1}(\text{generic point})\cap \pi_{13}^{-1}(\text{generic point})\cap (m' \text{ conditions})].$$ Line $\ell_1$ must satisfy the condition coming from $x_1$, which can be interpreted as saying that line $\ell_1$ lies in a fixed pencil of lines; however, this is inconsistent with the conditions requiring points $p_{12}$ and $p_{13}$ to lie in fixed (and general) positions. Hence $[M]m=0$.

For type 4 we just look at the first monomial $m= x_1^2y_{23}^2x_2y_{12}$. Then $[M]m$ equals $$[M \cap \pi_{1}^{-1}(\text{generic point})\cap \pi_{23}^{-1}(\text{generic point})\cap \pi_{2}^{-1}(\text{generic line})\cap \pi_{12}^{-1}(\text{generic line})].$$ In this case $\ell_1$ must be a fixed generic line and $p_{23}$ must be a fixed generic point. Since line $\ell_2$ must lie in a fixed pencil and since $p_{23} \in \ell_2$, the line $\ell_2$ is fixed too. Then $p_{12} = \ell_1 \cap \ell_2$ is fixed and so it cannot satisfy the generic linear condition in the class $[M]m$. Hence $[M]m=0$.

In this case we have fixed a generic line for line 1 and force $p_{23}$ to be fixed point in general position. The condition corresponding to $[\pi^{-1}_2(\text{generic line})]$ forces line 2 to lie in a fixed pencil. Then line 2 is determined (since it passes through $p_{23}$) and so is $p_{12} = \ell_1 \cap \ell_2$. However, the condition corresponding to $[\pi^{-1}_{12}(\text{generic line})]$ forces $p_{12}$ to lie on a fixed general line (that is, not through $\ell_1 \cap \ell_2$). Thus, we cannot find such a configuration and so $[M]m=0$.

Now we define another set of monomials and prove that each of these monomials have coefficient equal to 1 in $[M]$. Again we define this set as a union of a few different types. Unfortunately, there are many more types in this set even though there are less total monomials. We define the types in the following table.

\begin{longtable}{|c|l|}
\hline
Type & Monomials\\
\hline
\hline
$S_1$ & $\{x_1^2x_2^2x_3^2\}$\\
\hline
$S_2$ & $\{x_1^2x_2^2x_3y_{13},x_1^2x_2^2x_3y_{23},x_1^2x_3^2x_2y_{23},x_1^2x_3^2x_2y_{12}, x_2^2x_3^2x_1y_{13},x_2^2x_3^2x_1y_{12}\}$\\
\hline
$S_3$ & $\{x_1^2x_2^2y_{13}y_{23},x_1^2x_3^2y_{12}y_{23},x^2_2x_3^2y_{12}y_{13}\}$\\
\hline
$S_4$ & $\{ x_1^2x_2x_3y_{23}^2,x_2^2x_1x_3y_{13}^2,x_3^2x_2x_1y_{12}^2\}$\\
\hline
$S_5$ & $\{ x_1^2x_2x_3y_{12}y_{13},x_2^2x_1x_3y_{12}y_{23},x_3^2x_1x_2y_{13}y_{23}\}$\\
\hline
$S_6$ & $\{ x_1^2x_2x_3y_{23}y_{12},x_1^2x_2x_3y_{23}y_{13},x_2^2x_1x_3y_{13}y_{12},x_2^2x_1x_3y_{13}y_{23},$ \\
& $x_3^2x_2x_1y_{12}y_{13},x_3^2x_2x_1y_{12}y_{23}\}$\\
\hline
$S_7$ & $\{ x_1^2x_2y_{13}y_{23}^2,x_1^2x_3y_{12}y_{23}^2,x_2^2x_1y_{23}y_{13}^2,x_2^2x_1y_{23}y_{13}^2,x_3^2x_1y_{23}y_{12}^2,x_3^2x_2y_{13}y_{12}^2\}$\\
\hline
$S_8$ & $\{ x_1^2x_2y_{12}y_{13}y_{23},x_1^2x_3y_{12}y_{13}y_{23}, x_2^2x_1y_{12}y_{13}y_{23},x_2^2x_3y_{12}y_{13}y_{23},$\\
& $x_3^2x_2y_{12}y_{13}y_{23},x_3^2x_1y_{12}y_{13}y_{23}\}$\\
\hline
$S_9$ & $\{x_1^2y_{23}^2y_{12}y_{13},x_2^2y_{13}^2y_{23}y_{12},x_3^2y_{12}^2y_{23}y_{13}\}$ \\
\hline
$S_{10}$ & $\{ x_1x_2x_3y_{12}^2y_{13},x_1x_2x_3y_{12}^2y_{23},x_1x_2x_3y_{13}^2y_{12},x_1x_2x_3y_{13}^2y_{23},$\\
& $x_1x_2x_3y_{23}^2y_{12},x_1x_2x_3y_{23}^2y_{13}\}$\\
\hline
$S_{11}$ & $\{ x_1x_2y_{12}^2y_{13}y_{23},x_1x_3y_{12}y_{13}^2y_{23},x_3x_2y_{12}y_{13}y_{23}^2\}$\\
\hline
$S_{12}$& $\{x_1x_2y_{12}y_{13}^2y_{23},x_1x_2y_{12}y_{13}y_{23}^2,x_1x_3y_{12}^2y_{13}y_{23},x_1x_3y_{12}y_{13}y_{23}^2,$\\
& $x_3x_2y_{12}y_{13}^2y_{23},x_3x_2y_{12}^2y_{13}y_{23} \}$\\
\hline
$S_{13}$ & $\{ x_1x_2y_{13}^2y_{23}^2,x_1x_3y_{12}^2y_{23}^2,x_2x_3y_{12}^2y_{13}^2\}$\\
\hline
$S_{14}$ & $ \{ x_1y_{12}^2y_{13}y_{23}^2,x_1y_{12}y_{13}^2y_{23}^2, x_2y_{12}^2y_{13}^2y_{23},x_2y_{12}y_{13}^2y_{23}^2, $\\
& $x_3y_{12}^2y_{13}y_{23}^2,x_3y_{12}^2y_{13}^2y_{23}\}$\\
\hline
$S_{15}$ & $\{y_{12}^2y_{13}^2y_{23}^2\}$\\
\hline
\end{longtable}

Now put $Z_2=S_1\cup \cdots \cup S_{15}$. In the next table we compute the intersection of $M\cap m$ where $m\in Z_2$. We briefly describe, for each type, how this intersection results in a unique arrangement of three lines with the marked points $p_{ij}$. Again we use the argument that within each $S_i$ we only need to check one monomial since all the others within that class can be obtained from permuting coordinates.

\begin{longtable}{|c|l|l|}
\hline
Type &Monomial& Intersect this class with $M$ gives a unique point \\
\hline
\hline
$S_1$&$x_1^2x_2^2x_3^2$ &Fixing three lines fixes the three intersection points. \\
\hline
$S_2$ & $x_1^2x_2^2x_3y_{13}$ & Fixing lines 1 and 2. Then fixing a linear condition\\
&& point 13 gives a unique point and then putting \\
&&a linear condition on line 3 fixes it.\\
\hline
$S_3$ & $x_1^2x_2^2y_{13}y_{23}$ & Fixing lines 1 and 2. Then giving linear conditions on points\\
&& 13 and 23 fixes those points which fixes line 3.\\
\hline
$S_4$ & $x_1^2x_2x_3y_{23}^2$ & Fix line 1 and point 23. Then putting linear conditions on lines \\
&& 2 and 3 fixes them.\\
\hline
$S_5$ & $x_1^2x_2x_3y_{12}y_{13}$ & Fix line 1. Then Putting a linear condition on points \\
&&12 and 13 fixes them on line 1. \\
&&Then putting linear conditions on lines 2 and 3 fixes them.\\
\hline
$S_6$ &$x_1^2x_2x_3y_{12}y_{23}$ & Fixing line 1 and giving a linear condition on point 12\\
&& fixes point 12. Together with putting a linear condition on\\
&& line 2 fixes it. Then putting a linear condition on point 23\\
&& fixes it and line 3.\\
\hline
$S_7$ & $x_1^2x_2y_{13}y_{23}^2$ & Fixing line 1, point 23, and putting a linear condition on \\
&&line 2 fixes it. Then putting a linear condition on point 13\\
&& fixes it on line one and hence fixes line 3.\\
\hline
$S_8$ & $x_1^2x_2y_{12}y_{13}y_{23}$ & Fix line 1 and putting liner conditions on points \\
&&12 and 13 fixes these points. Then putting a linear\\
&& condition on line 2 fixes it. And then finally fixing a linear \\
&& condition on point 23 fixes it and line 3.\\
\hline
$S_9$ &$x_1y_{23}^2y_{12}y_{13}$ & Fix line 1 and point 23. Putting linear conditions \\
&&on points 12 and 13 fixes them on line 1. Then this \\
&& fixes lines 2 and 3.\\
\hline
$S_{10}$ &$x_1x_2x_3y_{12}^2y_{13}$ & Fixing point 12 with putting linear conditions \\
&& on lines 1 and 2 fixes them. Then putting a linear condition on \\
&& point 13 fixes it on line 1. Then putting a linear condition\\
&& on line 3 fixes it.\\
\hline
$S_{11}$&$x_1x_2y_{12}^2y_{13}y_{23}$ & Fixing point 12 with putting linear conditions \\
&& on lines 1 and 2 fixes them. Then putting linear conditions on\\
&& points 13 and 23 fixes them and hence line 3.\\
\hline
$S_{12}$ &$x_1x_2y_{12}y_{13}^2y_{23}$ & Fixing point 13 and putting a linear condition on \\
&& line 1 fixes it. Then putting a linear condition on point 12 fixes \\
&& it on line 1. Then the linear condition on line 2 fixes it.\\
&&Then the linear condition on point 23 fixes it and line 3.\\
\hline
$S_{13}$ &$x_1x_2y_{13}^2y_{23}^2$& Fixing points 13 and 23 fixes line 3. Then putting linear\\
&&conditions on lines 1 and 2 fixes them\\
\hline
$S_{14}$ & $x_1y_{12}^2y_{13}y_{23}^2$& Fixing points 12 and 23 fixes line 2. Then giving a linear\\
&&condition on line 1 fixes it. Then putting a linear condition\\
&& on point 13 fixes it on line 1 which also fixes line 3.\\
\hline
$S_{15}$&$y_{12}^2y_{13}^2y_{23}^2$ & The points 12, 13, and 23 are fixed. This fixes lines 1,2, and 3.\\
\hline
\end{longtable}

Again since the complement $c(Z_2)=Z_2$ we have shown that the coefficients of all these monomials in $[M]$ are equal to 1. There is just one more monomial to consider. Let $Z_3=\{x_1x_2x_3y_{12}y_{13}y_{23}\}$. This set also has the property $c(Z_3)=Z_3$. Now we compute its coefficient. The class $x_1x_2x_3y_{12}y_{13}y_{23}$ puts a linear condition on each line and point. For the lines we can think of this as fixing three points $q_1$, $q_2$, and $q_3$ in general position and requiring that $p_i \in \ell_i$, for $i\in \{1,2,3\}$. Similarly, to interpret the linear conditions on the points $p_{ij}$ we fix three lines $L_{12}$, $L_{13}$, and $L_{23}$ in general position, and require that $p_{ij} \in L_{ij}$.

Suppose now that we choose an arbitrary point $p_{12}$ on line $L_{12}$. Fix a embedding $\phi :\P^1\to \P^2$ defined by $\phi([a:b])=[f_1(a:b):f_2(a:b):f_3(a:b)]$ where $\phi(\P^1)=L_{12}$ and each $f_i$ is homogeneous of degree 1. Then there exists $[a:b]\in \P^1$ such that $\phi ([a:b])=p_{12}$. Now since the points $q_1$ and $p_{12}$ are fixed we have that line 1, $\ell_1$, is fixed. Line $\ell_1$ fixed means that we can get $p_{13}=\ell_1\cap L_{13}$. Next we find $\ell_3$ is determined by the points $p_{13}$ and $q_3$. With $\ell_3$ we can get $p_{23}=\ell_3\cap L_{23}$. Now we make $\ell_2$ the line through the points $p_{23}$ and $q_2$. Since we chose $p_{12}$ randomly at the beginning of this process we have that $P=\ell_2\cap L_{12}$ may not be equal to $p_{12}$. Each step in the above process to determine each line and point is a cross product computation with a vector of constants (the line or point that is fixed before the choice of $p_{12}$) with a vector of homogenous degree 1 in the functions $f_i$. Hence there exist linear functions $g_1$ and $g_2$ such that $P=\phi([g_1(a:b),g_2(a:b)])$. Now $P=p_{12}$ if and only if $$\det \left[ \begin{array}{cc} g_1(a:b)&g_2(a:b)\\
a&b
\end{array}\right] =0.$$ Since this a homogeneous quadratic polynomial of two variables there must be 2 solutions up to multiplicity. Hence the coefficient of the monomial $x_1x_2x_3y_{12}y_{13}y_{23}$ is 2.

The total number of terms in $A_6$ 141 (this can be calculated by a standard inclusion-exclusion argument). Since $|Z_1|=78$, $|Z_2|=62$, and $|Z_3|=1$ we have computed the coefficients of all possible terms. Finally with any computer software system one can expand the polynomial $\prod_{1 \leq i < j \leq 3} (x_i
+ y_{ij})(x_j + y_{ij})$ and see that the coefficients match what we have just determined. This completes the proof. \end{proof}

\begin{remark} A generic element of $M$ can be thought of as a generic arrangement of 3 lines. The complement inside $M$ of this generic locus consists of several 5 dimensional subvarieties. These consist of arrangements where two lines are the same and arrangements consisting of a triple line, together with their marked points. \end{remark}

Now we determine the characteristic numbers for 3 lines.
 
\begin{theorem} \label{cn3}
The characteristic numbers $\N_3(p,6-p)$ for a generic arrangement
of 3 lines are given in the following table.

\begin{center}
\begin{tabular}{|l||l|l|l|l|l|l|l|}
\hline $p$ & 0 & 1 & 2 & 3 & 4 & 5 & 6 \\ \hline

$\N_3(p,6-p)$ & 15 & 30 & 48 & 57 & 48 & 30 & 15 \\ \hline
\end{tabular}
\end{center}

\end{theorem}

\begin{proof}
If we intersect $M$ with 6 generic conditions then we will obtain a finite number of points none of which will lie on the 5-dimensional non-generic subvarieties. So, when we multiply the class $[M]$ with that from Lemma \ref{transpts} we will only be counting generic arrangements of 3 lines. From Lemma \ref{3linesM} we have $$[M] = [\cap_{1 \leq i < j \leq 3} (p_{ij} \in \ell_i) \cap (p_{ij}
\in \ell_j)] = \prod_{1 \leq i < j \leq 3} (x_i
+ y_{ij})(x_j + y_{ij}).$$ Then the condition that the arrangement
$\A$ passes through a given point $p \in \P^2$ has class $x_1+x_2+x_3$. Similarly
the condition that a given line $L \in \P^2$ contains one of the
$p_{ij}$ has class $y_{12}+y_{13}+y_{23}$. Using the Moving Lemma, Lemma \ref{transpts}, and Remark \ref{tangencyremark} the degree of the class
$[M](x_1+x_2+x_3)^p(y_{12}+y_{13}+y_{23})^{6-p}$ in $A_*\left( \left(\P^{2*}\right)^3\times \left(\P^{2}\right)^3\right)$ counts the number of labeled 3-generic arrangements that pass
through $p$ general points and are tangent to $6-p$ general lines in
$\P^2$. To remove the effect of the labeling, we divide these numbers
by $3!$ to obtain the characteristic numbers. \end{proof}

\begin{remark}
The symmetry of $\N_3(p,6-p)$ is easily explained: the dual of a
3-generic arrangement through $p$ points $P_1,\ldots,P_p$ and tangent
to $6-p$ lines $L_1,\ldots,L_{6-p}$
is a 3-generic arrangement through the $6-p$ dual points
$\hat{L}_1,\ldots,\hat{L}_{6-p}$ and tangent to the $p$ dual lines
$\hat{P}_1,\ldots,\hat{P}_p$. This symmetry does not occur in the 4 lines case as one can see below.\end{remark}

Now we consider the 4 line case. We compute the class $M$ as before but we use transversality arguments in place of the method of undetermined coefficients.

\begin{lemma}\label{4linesM}
The set
$$ M = \left\{(\ell_1,\dots,\ell_4,p_{12},\dots,p_{34}) \in
\left( \P^{2*}\right)^{4}\times \left( \P^{2}\right)^{6} \; : \; p_{ij} \in
\ell_i \cap \ell_j \text{ for $1 \leq i < j \leq 4$} \right\}$$
is a quasi-projective variety of codimension 12 in $\left( \P^{2*}\right)^{4}\times \left( \P^{2}\right)^{6}$. The Chow ring of the ambient space is $$A=A\left(\left( \P^{2*}\right)^{4}\times \left( \P^{2}\right)^{6}\right) \cong \Z[x_1,\dots, x_4,y_{12},\dots , y_{34}]/(x_1^3,\dots,x_4^3,y_{12}^3,\dots,y_{34}^3)$$ where the $x_i$ correspond to the lines and the $y_{ij}$ correspond to the intersection points $p_{ij}$. In this ring the class of $M$ is $$[M]=\left[\cap_{1 \leq i < j \leq 4} \left( (p_{ij} \in \ell_i) \cap (p_{ij}
\in \ell_j)\right)\right] = \prod_{1 \leq i < j \leq 4} (x_i
+ y_{ij})(x_j + y_{ij}) .$$

\end{lemma}

\begin{proof}

As in Lemma \ref{3linesM} set $\pi_i : \left( \P^{2*}\right)^4 \times \left( \P^2 \right)^6\to \P^{2*}$ to be the projection to the factor corresponding to the line $\ell_i$ and $\pi_{ij} : \left( \P^{2*}\right)^4 \times \left( \P^2 \right)^6\to \P^{2}$ to be the projection to the factor corresponding to the point $p_{ij}$. The condition $p_{ij}\in \ell_i$ is a bilinear hypersurface in $\left( \P^{2*}\right)^{4}\times \left( \P^{2}\right)^{6}$ whose class is $x_i+y_{ij}$. The closure of the variety $M$ is the intersection of these 12 hypersurfaces. We will show that this is a local complete intersection.

Put $\pi=\pi_1\times \pi_2\times \pi_3\times \pi_4$. Examine the restriction $\pi :M \to \left(\P^{2*}\right)^4$. Now stratify $\left(\P^{2*}\right)^4$ into the following quasi-projective subvarieties $G=$ ``generic arrangements '', $T=$ ``arrangements where two lines are equal and the other two are generic'', $D=$ ``arrangements where there are two generic double lines'', $H=$``arrangements where three lines are equal and the last one is generic'', and $F=$ ``arrangements where all four lines are equal''. So, $\left(\P^{2*}\right)^4=G\cup D\cup T\cup H\cup F$ and all the unions are disjoint. The dimension of $G$ is 8 and for any $x\in G$ the fiber $\pi^{-1}(x)$ is a unique point. Hence $\pi^{-1}(G)$ is 8-dimensional. The dimension of $T$ is 6 and for any $x\in T$ the fiber $\pi^{-1}(x)$ is 1-dimensional. Hence $\pi^{-1}(T)$ is 7-dimensional. The dimension of $D$ is 4 and for $x\in D$ the fiber $\pi^{-1}(x)$ is 2-dimensional. Hence $\pi^{-1}(D)$ is 6-dimensional. The dimension of $H$ is 4 and for $x\in H$ the fiber $\pi^{-1}(x)$ is 3-dimensional. Hence $\pi^{-1}(H)$ is 7-dimensional. The dimension of $F$ is 2 and for $x\in F$ the fiber $\pi^{-1}(x)$ is 6-dimensional. Hence $\pi^{-1}(F)$ is 8-dimensional. Since $M=\pi^{-1}(G)\cup\pi^{-1}(T)\cup\pi^{-1}(D)\cup\pi^{-1}(H)\cup\pi^{-1}(F)$ we have shown that $M$ is 8-dimensional.

Hence $M$ is a local complete intersection in $\left( \P^{2*}\right)^4 \times \left( \P^2 \right)^6$. Using Theorem 5.10 in \cite{EH} the class of the intersection $[M]$ is the product of the classes $x_i+y_{ij}$.
\end{proof}

\begin{theorem} \label{cn4}
The characteristic numbers $\N_4(p,8-p)$ for a generic arrangement
of 4 lines are given in the following table.

\begin{center}
\begin{tabular}{|l||l|l|l|l|l|l|l|l|l|}
\hline $p$ & 0 & 1 & 2 & 3 & 4 & 5 & 6 & 7 & 8 \\ \hline

$\N_4(p,8-p)$
& 16695 & 17955 & 13185 & 8190 & 4410 & 2070 & 855 & 315 & 105
\\ \hline
\end{tabular}
\end{center}

\end{theorem}

\begin{proof}

We use an argument similar to the proof of Theorem \ref{cn3} but there is a subtlety. The component $F$ from Lemma \ref{4linesM} consisting of quadruple lines has the same dimension as the component $G$ of generic arrangements of 4 lines. It follows that $[M] = a[F] + b[G]$ for some integers $a$ and $b$. Specialize the equations that cut out $M$ to a fixed generic arrangement $\A$ of 4 lines to see that there is just one element of $M$ lying over $\A$, with multiplicity one; this shows that $a = 1$. Similarly, if we specialize the equations that cut out $M$ to $\A$, a fixed quadruple line with six fixed marked points $p_{ij}$, we again see that the fiber of $M$ lying over $\A$ is a reduced point in $\M$; so $b$ = 1 and $[M] = [F] + [G]$. Let $C = (x_1+x_2+x_3+x_4)^p(y_{12}+y_{13}+y_{14}+y_{23}+y_{24}+y_{34})^{8-p}$. The degree of $[M]C = [F]C + [G]C$ also counts some instances in which all the $\ell_i$ are equal (coming from the term $[F]C$). For
instance, in the case where $p=0$, we seek elements in $M$ with some
$p_{ij}$ on each of $8$ lines $L_1, \ldots, L_8$. If we set all the
$\ell_i$ equal to the line joining $L_i\cap L_j$ to $L_s \cap L_t$
then the common line $\ell_i$ meets the 8 lines $L_j$ in just
$\binom{4}{2}=6$ points. Labeling these points with the six points
$p_{ij}$ gives an element of $M$ that is included in our degree count
but that doesn't correspond to a 4-generic arrangement; this is
illustrated in the left picture of Figure \ref{4lines}, where the
thick line represents the common line $\ell_i$. There are
$\binom{8}{2,2,4}6!/2$ ways to produce such examples. Removing these
from the degree count, we obtain the degree of $[G]C$ and dividing the result by 4! to account for the
effect of labeling, we obtain the value of $N_4(0,8)$ in the Table.

\begin{figure}
\begin{center}
\scalebox{1.05}{\includegraphics{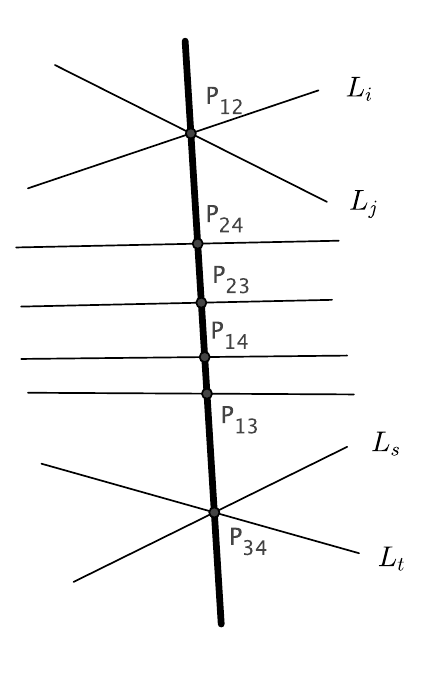}}
\scalebox{1.05}{\includegraphics{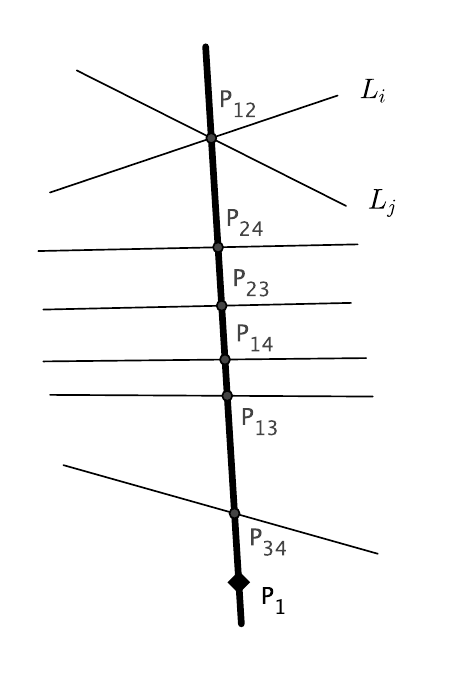}}
\scalebox{1.05}{\includegraphics{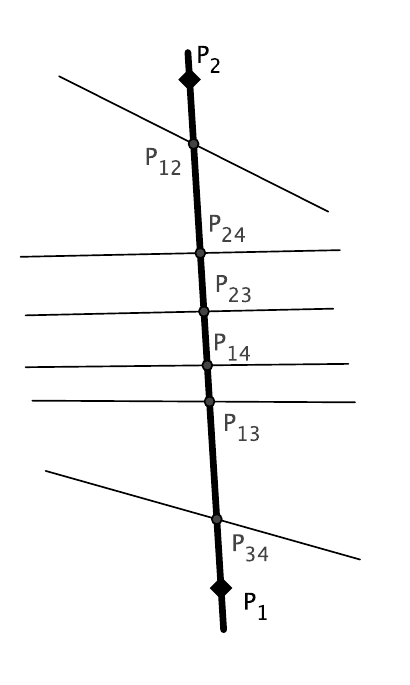}}
\end{center}
\caption{Quadruple lines counted by the degree computation
for $p=0$ (left), $p=1$ (middle), $p=2$ (right).}
\label{4lines}
\end{figure}

Similarly, quadruple lines are counted in the degree
computations for $p=1$ and $p=2$, as illustrated in the middle and
right pictures of Figure \ref{4lines}. This requires an adjustment of
$\binom{8}{2,6}6!$ in the case of $p=1$ and $6!$ in the case of
$p=2$, giving rise to the values for $N_4(p,8-p)$ in the statement of the theorem.
\end{proof}

\begin{remark} It would be interesting to have an explanation for the unimodality of
the characteristic numbers $N_k(p,2k-p)$. In
particular, do the characteristic numbers
$\{ N_k(p,2k-p)\}$ for generic arrangements of $k$ lines always form a unimodal sequence?
\end{remark}

\begin{remark} \label{braidr}
Taking the dual of the $4$-generic arrangement yields a braid line
arrangement $\mathcal{A}_3$, an arrangement of $6$ lines meeting in 4 triple points
and 3 double points, pictured in Figure \ref{braid}. This is related
to the braid arrangement (the reflecting hyperplanes of the
action of $\S_4$ on the variables of $\C^4$), $\A_3 = \{x_i = x_j \; : \; 1 \leq i
< j \leq 4\}$ in $\C^4$ which consists of 6 hyperplanes, each
containing the line $\{x_1=x_2=x_3=x_4\} \subset \C^4$. Quotienting
out the common line and projectivizing gives the braid arrangement in $\P^2$.

\begin{figure}[h!t]
\begin{center}
\scalebox{0.5}{\includegraphics{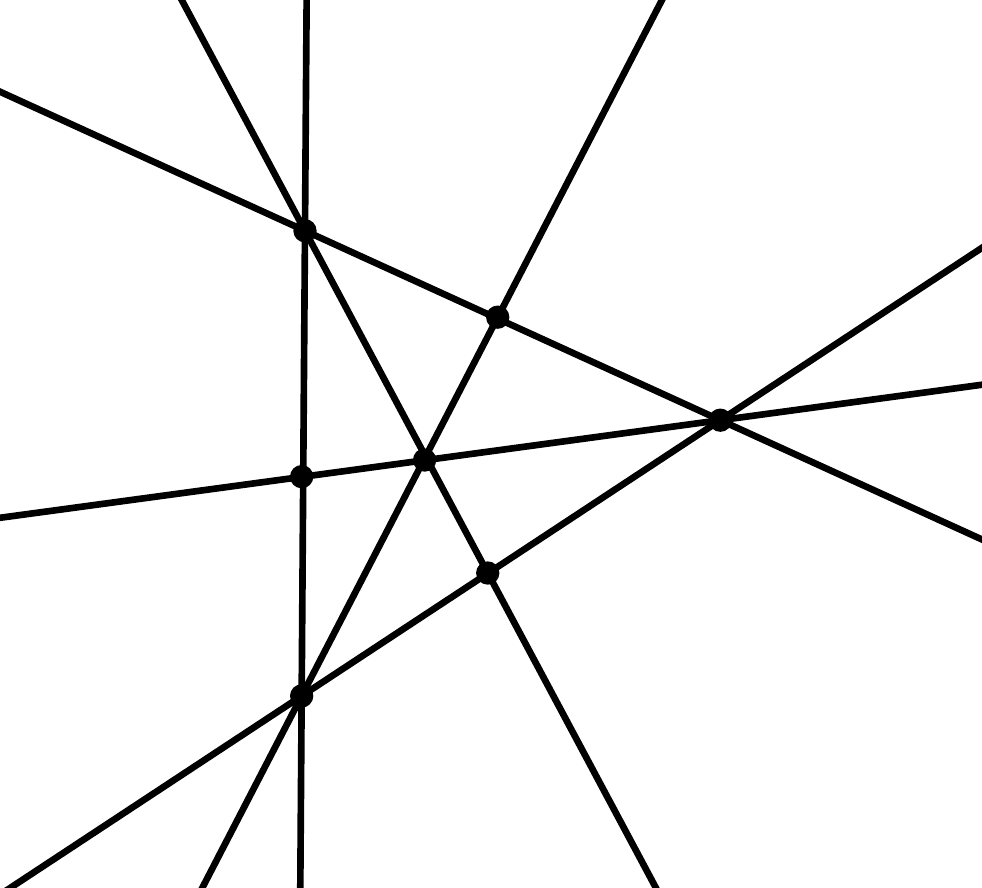}}
\end{center}
\caption{The braid arrangement $\mathcal{A}_3$ } \label{braid}
\end{figure}

The moduli space $\M_{\mathcal{A}_3}$ has the same dimension, 8, as the
moduli space of the dual of $\mathcal{A}_3$. The characteristic numbers
$N_{\mathcal{A}_3}(p,8-p)$ counting the number of arrangements
combinatorially equivalent to the braid arrangement that pass through
$p$ points and are tangent to $8-p$ lines in general position are
given by $N_4(8-p,p)$. For example, there are 16695 braid arrangements
through 8 points in general position in $\P^2$, a result initially
reported in Paul \cite{Paul}.
\end{remark}

The characteristic numbers are important because they determine the
answer to all enumerative questions involving points and tangency to
arbitrary smooth curves (see \cite[section 10.4]{Fulton} and the
references therein for the history of this result).

\begin{theorem} \label{FKMt}
Let $\L_\A$ be the intersection lattice of a line arrangement and let
$t = \dim \MLA$. The number of line arrangements with intersection
lattice isomorphic to $\L_\A$ through $p$ points and tangent to $t-p$
smooth curves of degrees $n_1, \ldots, n_{t-p}$ and classes\footnote{The class of a curve is the number of lines passing
through a given general point and tangent to the curve at a simple
point. For example, the class of a smooth curve of degree $d$ is $d(d-1)$.} $m_1,
\ldots, m_{t-p}$ in general position is obtained from the
product $$ \mu^p \prod_{i=1}^{t-p} (m_i \mu + n_i\nu) $$ by expanding the
polynomial and replacing the monomial $\mu^k \nu^{t-k}$ by the
associated characteristic number -- the number
of arrangements with lattice type $\L_\A$ passing through $k$ general
points and tangent to $t-k$ lines.
\end{theorem}

\begin{proof}
The proof follows immediately from the argument due to Fulton
and MacPherson in \cite[section 10.4]{Fulton} if we interpret each
point condition as a tangency condition to a curve of degree 0 and
class 1. Alternatively, we can work with the dimension $(t-p)$ family of arrangements in $\M_\A$ that pass through the specified $p$ points, apply Fulton and MacPherson's theorem and then re-interpret the results in terms of the polynomial displayed above.
\end{proof}

\begin{example}
To compute the number of generic arrangements of 4 lines through 3
points, tangent to a given line and tangent to four smooth conics, we
expand the product $\mu^3(2\mu + 2\nu)^4\nu^1$ and replace the
monomials with the appropriate quantities from Theorem \ref{cn4}. The
answer
is $$2^4(315)+2^4\binom{4}{1}(855)+2^4\binom{4}{2}(2070)+2^4\binom{4}{3}(4410)+
2^4(8190) = 671760. $$
\end{example}

\begin{remark}
  The
  computation of $N_k(p,2k-p)$ becomes more intricate as $k$ increases. For instance,
  using the method of proof from Theorem \ref{cn4} to compute
  $N_5(0,10)$, we see that every quintuple line determines a point in
  the incidence correspondence $M$. This higher-dimensional component
  of the solution space (isomorphic to $\P^2$) ``counts'' as a finite number of points in the
  standard intersection theory computation of $N_5(0,10)$. The
  appropriate count is called the excess intersection of this
  component and can be computed using the Excess
  Intersection Formula \cite[section 6.3]{Fulton}. We do not pursue
  this approach further here.

\end{remark}

\section{Pencils and Cones Over Generic Arrangements}
\label{section:dconed}

An arrangement of lines is said to be a pencil if all lines pass
through a common point. Such an arrangement is also a $0$-coned
generic arrangement. We start this section with a simple enumerative
result about pencils.

\begin{theorem}
If $\A$ is a pencil of $k$ lines in $\P^2$ then $\dim \MLA = k+2$ and
the characteristic numbers are $\N_\A(k+2,0) = 3\binom{k+2}{4}$,
$\N_\A(k+1,1) = \binom{k+1}{2}$, $\N_\A(k,2)=1$. All other
characteristic numbers are 0.
\end{theorem}

\begin{proof}
  By the Pigeonhole Principle, if a pencil $\A$ of $k$ lines passes
  through more than $k+3$ points in general position then 3 of the
  lines must each contain 2 points and all three lines must be
  coincident, but this means that these 6 points must satisfy an
  algebraic relation and thus violates the assumption that the points
  were in general position. As well, it is easy to see that there
  are a finite number of pencils $\A$ of $k$ lines passing through $k+2$ points
  in general position (two lines contain 2 points each and intersect
  in the node of the pencil, the other $k-2$ lines each contain one of
  the remaining points) so $\dim \MLA = k+2$. There are $N_\A(k+2,0) =
\binom{k+2}{2,2,k-2}/2 = 3 \binom{k+2}{4}$ ways to create such an
  arrangement.

  To compute $N_\A(k+1,1)$ note that the node of the pencil must lie
  on the given line. The arrangement must have one line containing $2$
  of the $k+1$ points and this line intersects the given line at the
  node of the pencil; the remaining $k-1$ lines in the pencil each
  contain 1 point and the node. There are $N_\A(k+1,1) = \binom{k+1}{2}$ ways to
  choose the first line in the pencil; the rest of the construction is determined. Finally, to compute
  $N_\A(k,2)$ note that the node must lie at the intersection of the
  two given lines and that each line in the pencil must pass through
  the node and one of the $k$ given points. So $\A$ is uniquely
  determined. \end{proof}

We proceed to study $d$-coned generic arrangements with $d\geq 1$. Recall
from Definition \ref{dconed} that an arrangement $\A$ of $k>n$
hyperplanes in $\P^n$ is a $d$-coned generic arrangement if $\A$ can
be obtained from a generic hyperplane arrangement $\mathcal{B}$ in a
linear subspace $H \cong \P^{n-d-1}$ by taking the cone with a
$d$-dimensional linear space (equivalently, with $d+1$ general
points). Given a $d$-coned generic arrangement $\A$ in $\P^n$, the
intersection with a general linear space $H$ of dimension $n-d-1$ is a
generic hyperplane arrangement in $\P^{n-d-1}$.

\begin{lemma} \label{dimlemma} Let $\A$ be a $d$-coned generic
  arrangement of $k$ hyperplanes in $\P^n$. Then the dimension of the
  moduli space $\M_\A$ is $D = (d+1)(n-d) + k(n-d-1)$.
\end{lemma}

\begin{proof}
Note that an arrangement $\mathcal{B} \in \M_\A$ is determined by $\Lambda$ in
the Grassmannian $\G(d,n)$ of $d$-dimensional linear spaces of
$\P^{n}$, together with $k$ points in $\P(\C^{n+1}/\Lambda) \cong
\P^{n-d-1}$ (since $\Lambda$ determines a subspace of dimension $d+1$
in $\C^{n+1}$). So the dimension of $\M_\A$ is $D = \text{dim } \left[ \G(d,n)
\times (\P^{n-d-1})^k \right] = (d+1)(n-d) + k(n-d-1)$.
\end{proof}

As a warm-up, we count the number of $0$-coned generic arrangements through a set of
points.

\begin{theorem} \label{0pencil}
Let $\A$ be a $0$-coned generic arrangement of $k \geq n$ hyperplanes in $\P^n$. Then the number of arrangements combinatorially
equivalent to $\A$ that pass through $kn+n-k$ points in general
position is $$\N_\A = \frac{1}{k!} \binom{k}{n}
\binom{kn+n-k}{(n)^n,(n-1)^{k-n}} = \frac{(kn+n-k)!}{(n!)^{n+1}((k-n)!)((n-1)!)^{k-n}}. $$
\end{theorem}

\begin{proof}
By Lemma \ref{dimlemma}, $\dim \MLA = kn+n-k$. The only way to obtain
an arrangement combinatorially equivalent to $\A$ through the $kn+n-k$
points is to have $n$ hyperplanes pass through $n$ of the points and
meet in a new point $p$; then each of the remaining $k-n$ of the hyperplanes
pass through the point $p$ and $n-1$ of the remaining given points. Labeling the
set of hyperplanes with $n$ points as $H_1, \ldots, H_n$ and the
remaining hyperplanes $H_{n+1},\ldots, H_k$, we have
$\binom{kn+n-k}{(n)^n,(n-1)^{k-n}}$ labeled arrangements. Dividing by
$n! k-n!$ gives the number of unlabeled arrangements and rearranging
terms gives the desired formula.
\end{proof}

It is tempting to think that the number of $d$-coned generic
arrangements of $k$ hyperplanes in $\P^n$
that pass through $D =
(d+1)(n-d) + k(n-d-1)$ points in general position is just $$\N_\A = \frac{1}{k!} \binom{k}{n}
\binom{D}{(n)^{n-d},(n-d)^{k-n+d}}, $$
but in fact this grossly undercounts the arrangements. In this sense,
the $0$-coned generic arrangement result in Theorem \ref{0pencil} is misleading.

Before stating our result on the enumeration of $d$-coned generic arrangements,
 we quickly review the structure of the intersection ring $A=A(\G(d,n))$ of the
Grassmannian $\G(d,n)$ of $d$-planes in $\P^n$. Given any complete
flag of linear spaces $F: F_0 \subsetneq F_1 \subsetneq \ldots
\subsetneq F_{n} = \P^n$ with $\dim F_j = j$, we have a cellular
decomposition of $\G(d,n)$ into affine cells; for each $(d+1)$-tuple
$(\alpha_0, \alpha_1, \ldots, \alpha_{d})$ of
non-increasing non-negative integers $\alpha_i \leq n-d$ we have an
affine cell $X_{\alpha}(F)$ of codimension $| \alpha | = \sum \alpha_i$. To describe the
$d$-planes $\Lambda$ in $X_{\alpha}(F)$, note that for each $d$-plane
$\Lambda$ the sequence $\dim \Lambda \cap F_j$ is completely
characterized by the $d+1$ places where the dimension increases, the
so-called rank jumps. Now $\Lambda \in X_{\alpha}(F)$ precisely when
the collection of rank jumps is $\{n-d+j-\alpha_j \; : \; 0 \leq j \leq
d\}$. For instance the largest open cell $X_{(0)^{d+1}}$ consists of
$d$-planes in general position with respect to the flag $F$. See
Hatcher \cite{Hatcher} for more details on the cellular decomposition
of $\G(d,n)$. Define the intersection class $\sigma_\alpha \in A$ to be the class of the closure of $X_\alpha(F)$. The set of
$\sigma_\alpha$ form an additive basis for the intersection ring
$A(\G(d,n))$. The multiplicative structure of the intersection ring
is described by the Pieri and Giambelli formulas (see e.g. Gatto
\cite{GattoASC,GattoIMPA} for statements of these results and an
interesting description of the multiplicative structure of
$A(\G(d,n))$ in terms of Hasse-Schmidt derivations). The class
$\sigma_{(0)^{d+1}}$ is the multiplicative identity in $A(\G(d,n))$. We denote
the tuple consisting of $i$ 1's followed by $d+1-i$ 0's by
$1^i$. Similarly, the tuple consisting of $d+1$ copies of $n-d$ is
denoted $(n-d)^{d+1}$. The class $\sigma_{(n-d)^{d+1}}$ is the class of a point and
the class $\sigma_{1^{d+1}}$ is the class of the set of $d$-planes
contained in $F_{n-1} \cong \P^{n-1}$. Though we will not need the
full power of the Pieri and Giambelli formulas, we note that there is
a duality between the classes $\sigma_\alpha$ in the sense that
$$\sigma_\alpha \cdot \sigma_{r((n-d)^{d+1}-\alpha)} =
\sigma_{(n-d)^{d+1}},$$ where $r$ is the operation that reverses a
tuple: $r(\beta_0,\ldots,\beta_{d}) =
(\beta_d,\ldots,\beta_0)$. As well, if $\sigma_{\alpha_1}, \ldots,
\sigma_{\alpha_t}$ are intersection classes with
$\alpha_{1}+\cdots+\alpha_t = \dim \G(d,n) = (n-d)(d+1)$ then their
product has a well-defined degree, denoted $\int_{\G(d,n)}
\sigma_{\alpha_1} \cdots \sigma_{\alpha_t}$, which can be interpreted
as the number of $d$-planes in $X_{\alpha_1} \cap \cdots \cap
X_{\alpha_t}$ if the classes $X_{\alpha_i}$ meet transversally (here
the $X_{\alpha_i}$ can be interpreted as coming from different flags,
in general position with respect to each other).

\begin{theorem} \label{countdpencil} If $\A$ is a $d$-coned generic
  arrangement of $k$ hyperplanes in $\P^n$, then the number of pencils
  combinatorially equivalent to $\A$ that pass through $D=(d+1)(n-d) +
k(n-d-1)$ points in general position is given by
$$ \N_\A = \frac{1}{k!} \sum_{\Gamma}
\sigma^{s} \binom{k}{s_0,s_1,\ldots,s_{d+1}}
\binom{D}{(n-(d+1))^{s_{0}},(n-d)^{s_{1}},\ldots,(n)^{s_{d+1}}}, $$
where $$\Gamma = \left\{(s_0,\ldots,s_{d+1}) \in \NN^{d+2}: \sum_{i=0}^{d+1} is_i =
(d+1)(n-d) = \dim \G(d,n),
\sum_{i=0}^{d+1} s_i = k\right\}$$ and $$\sigma^{s} = (\sigma_0)^{s_0}
(\sigma_1)^{s_1} (\sigma_{(1)^2})^{s_2} \cdots
(\sigma_{(1)^{d+1}})^{s_{d+1}},$$
using the notation of Remark \ref{notation}.
\end{theorem}

\begin{proof} We first note that the intersection ring $A(\G(d,n) \times (\P^{n*})^k)$ is just the tensor product of $A(\G(d,n))$ with $A((\P^{n*})^k)$. Now we determine the class $[Z] \in A(\G(d,n) \times \P^{n*})$, where $Z$ is
the closed set determining an incidence correspondence,
$$ Z = \{(\Lambda, H) \in \G(d,n) \times \P^{n*} \; : \; \Lambda \subset
H \}.$$ Note that $\G(d,n) \times \P^{n*}$ has dimension $(d+1)(n-d)+n$ and $Z$ has
dimension $(d+1)(n-d)+(n-d-1)$ so $Z$ has codimension $d+1$. It
follows that
$$[Z] = \sum_{\{\alpha: \; |\alpha| \leq d+1\}} a_\alpha
\sigma_{\alpha} h^{d+1-|\alpha|},$$
where $\sigma_{\alpha}$ and $h$ denote the pullbacks of these classes
to $A(\G(d,n) \times \P^{n*})$ and the $a_\alpha$ are integers.

We multiply $[Z]$ by $\sigma_{r((n-d)^{d+1}-\alpha)} h^{n-d-1+|\alpha|}$
to obtain $a_\alpha \sigma_{(n-d)^{d+1}}h^n$, representing $a_\alpha$
elements $(\Lambda,H) \in Z$ that also meet the flag $F$ in a manner
prescribed by the class $\sigma_{r((n-d)^{d+1}-\alpha)}
h^{n-d-1+|\alpha|}$. That is, $\Lambda$ has rank jumps at positions
$n-d-(n-d-\alpha_d) = \alpha_d, 1+\alpha_{d-1}, \ldots, d+\alpha_0$
(or earlier) and we can assume that $H$ must pass through $n-d-1+|
\alpha | $ points in general position.

Since $| \alpha | \leq d+1$ and $\alpha \in \NN^{d+1}$, we see
that if any $\alpha_i > 1$ then there are at least $d- | \alpha
| +2$ trailing zeros in $\alpha$. It follows that the first set of
rank jumps occur at positions $0,1,2, \ldots, d-| \alpha | +
1$. Now $H \supset \Lambda \supset F_{d-| \alpha | + 1}$, which
forces $H$ to pass through $d-| \alpha | + 2$ points in general
position in $F_{d-| \alpha | + 1}$. Together with the previous
$n-d-1+| \alpha | $ points in $H$, this forces $H$ to pass
through $n+1$ points in general position in $\P^n$. Of course this is
not possible, so $a_\alpha = 0$ if any entry of $\alpha$ is greater
than $1$.

On the other hand, if $\alpha = (1)^{d+1-\ell}$ then we multiply $[Z]$ by
$\sigma_{(n-d)^{\ell}(n-d-1)^{d+1-\ell}} h^{n-\ell}$ to get $a_{(1)^{d+1-\ell}}
\sigma_{(n-d)^{d+1}}h^n$, representing $a_{(1)^{d+1-\ell}}$ elements
    $(\Lambda,H)\in Z$ such that $\Lambda$ has rank jumps at positions
    in $0, 1, \ldots, \ell-1, \ell+1, \ldots, d+1$ (or earlier) and $H$
    contains $n-\ell$ points in general position. But then $H \supset
\Lambda \supset F_{\ell-1}$ and together with the conditions imposed
    by the $n-\ell$ points, $H$ is uniquely determined. Then $\Lambda = H
\cap F_{d+1}$ is also uniquely determined so $a_\alpha =
a_{(1)^{d+1-\ell}}=1$ (in the case $\ell=d+1$, i.e. $\alpha = (0)^{d+1}$,
    $\Lambda = F_d$ and $H$ is uniquely determined by $H \supset
\Lambda = F_d$ and the point conditions). It follows that the
    class $[Z] \in A(\G(d,n)\times \P^{n*})$ is
$$[Z] = \sum_{\ell=0}^{d+1} \sigma_{(1)^{d+1-\ell}} h^{\ell}. $$

Now we return to our main result and determine the number of
hyperplane arrangements consisting of $k$ hyperplanes in $\P^n$ that
all contain a $d$-dimensional linear space and pass through $D$ points
in general position. Let $$ Z_i = \{(\Lambda, H_1, \ldots, H_k) \in
\G(d,n) \times (\P^{n*})^k \; : \; \Lambda \subset H_i \} $$ and
write $h_i$ and $\sigma_\alpha$ for the pull-backs of the obvious
classes to $\G(d,k) \times (\P^{n*})^k$. Note that $[Z_i]
=\sum_{\ell=0}^{d+1} \sigma_{(1)^{d+1-\ell}} h_i^{\ell} \in A( \G(d,n) \times
(\P^{n*})^k)$ and $Z_i$ is a local complete intersection.

Now we show that $\bigcap_{i=1}^kZ_i$ is a local complete intersection. The dimension of the ambient space is $\dim \G(d,n) + nk$ and the dimension of $\bigcap_{i=1}^kZ_i$ is $\dim \G(d,n) + (n-(d+1))k$, showing that $\bigcap_{i=1}^kZ_i$ has codimension $k(d+1)$. Since each $Z_i$ has codimension $d+1$ and is a local complete intersection, their intersection is also a local complete intersection. Again using Theorem 5.10 of \cite{EH}, we have that on the level of classes $$\left[ \bigcap_{i=1}^kZ_i\right] =[Z_1][Z_2]\cdots [Z_k].$$

Let $P_1,\ldots, P_D$ be points in general position in $\P^n$ and for
each $P_i$ let $$ Y_i = \{(\Lambda, H_1,\ldots, H_k)
\in \G(d,n) \times (\P^{n*})^k \; : \; P_i \in H_1 \cup \cdots \cup H_k \}. $$ Note that $[Y_i] =
h_1+\cdots+h_k.$ Again from Lemma \ref{transpts} and using Kleiman's Transversality Theorem (see Kleiman's fundamental paper \cite{Kleiman}) we have that at the level of classes $$\left[\left( \cap_{i=1}^k Z_i \right) \cap \left( \cap_{j=1}^{D} Y_j
\right) \right] = \left[\left( \cap_{i=1}^k Z_i \right) \right] \left[\left( \cap_{j=1}^{D} Y_j
\right)\right] = [Z_1][Z_2]\cdots [Z_k][Y_1][Y_2]\cdots [Y_D].$$

Now
$$ \begin{array}{lll} \left[\left( \cap_{i=1}^k Z_i \right) \cap \left( \cap_{j=1}^{D} Y_j
\right)\right] & = & [Z_1] \cdots [Z_k] [Y_1] \cdots [Y_D] \\
& = & \left[ \prod_{i=1}^k \left( \sum_{\ell=0}^{d+1} \sigma_{(1)^{d+1-\ell}} h_i^{\ell} \right) \right] \left(
h_1+\cdots+h_k \right)^{D} \\
& = & \sum_{\Gamma}
\sigma^{s} \binom{k}{s_0,s_1,\ldots,s_{d+1}}
\binom{D}{(n-(d+1))^{s_{0}},(n-d)^{s_{1}},\ldots,(n)^{s_{d+1}}},
\end{array} $$
where $\Gamma = \{(s_0,\ldots,s_{d+1}) \in \NN^{d+2}: \sum_{i=0}^{d+1} is_i =
(d+1)(n-d) = \dim \G(d,n),
\sum_{i=0}^{d+1} s_i = k\}$ and $$\sigma^{s} = (\sigma_0)^{s_0}
(\sigma_1)^{s_1} (\sigma_{(1)^2})^{s_2} \cdots
(\sigma_{(1)^{d+1}})^{s_{d+1}}.$$
Since each unordered arrangement of hyperplanes gives $k!$
arrangements of labeled hyperplanes, we divide by $k!$ to obtain the
correct degree $\N_\A$. Since the points $P_i$ are in general position,
Kleiman's transversality theorem assures us that $\left( \cap_{i=1}^k Z_i \right) \cap \left( \cap_{j=1}^{D} Y_j
\right)$ is a finite collection of points, each appearing with multiplicity one. So we can interpret the above computation as saying that there are $N_\A$ $d$-coned generic arrangements of $k$ hyperplanes in $\P^n$ that pass through $D$ points in general position.
\end{proof}

It would be interesting to study the enumerative geometry of
hyperplane arrangements in a combinatorial equivalence class different
from the generic arrangements and $d$-coned generic arrangements. In general this seems
to require some subtle intersection theory computations. In
particular, it would be interesting to see if intersection theory
computations on the moduli space of stable maps can deal with the
enumerative geometry of more general families of line arrangements.

\begin{remark}
  The result in Theorem \ref{countdpencil} takes a particularly nice
  shape when $d=1$. In this case $\sigma^{s} = (\sigma_0)^{s_0}
(\sigma_1)^{s_1} (\sigma_{11})^{s_2}$ counts the number of lines
  that are in $s_2$ general hyperplanes in $\P^n$ and meet $s_1$
  general codimension-$2$ planes. Cutting down by the hyperplanes,
  this is the number of lines in $\P^{n-s_2}$ that meet $s_1 =
2(n-1)-2s_2$ codimension-$2$ planes. This was one of the classical
  problems studied by Schubert \cite{Schubert}, who showed that
  $\sigma^s = C_{n-s_2}$, where $C_N = \frac{1}{N}\binom{2N-2}{N-1}$
  defines the sequence of Catalan numbers (with the index shifted to
  start at 1).
\end{remark}

\begin{example}
The number of $1$-coned generic arrangements of $9$ hyperplanes in $\P^5$ that go through $35$ points is
$$ \begin{array}{lll} N_\A & = & \frac{1}{9!} \left[
\binom{9}{5,0,4}\binom{35}{(3)^5(4)^0(5)^4} +
\binom{9}{4,2,3}\binom{35}{(3)^4(4)^2(5)^3} +
2\binom{9}{3,4,2}\binom{35}{(3)^3(4)^4(5)^2} +\right. \\
& & \phantom{\frac{1}{9!} [[} \left. 5\binom{9}{2,6,1}\binom{35}{(3)^2(4)^6(5)^1} +
14\binom{9}{1,8,0}\binom{35}{(3)^1(4)^8(5)^0} \right] \\
& = & 148,467,792,706,702,950,173,442,750.
\end{array} $$
\end{example}

\bibliographystyle{plain}
\bibliography{counting}
\end{document}